\documentclass[11pt,twoside]{article}

\usepackage{amsmath,amsthm,amssymb}
\usepackage{a4wide}

\newtheorem{definition}{Definition}
\newtheorem{theorem}{Theorem}
\newtheorem{lemma}{Lemma}
\newtheorem{remark}{Remark}

\begin{document}

\title{Variable order Mittag--Leffler fractional operators on isolated time scales
and application to the calculus of variations\thanks{This is a preprint of a paper 
whose final and definite form is with Springer, as a chapter book.}}

\author{Thabet Abdeljawad$^{a}$\\
\texttt{tabdeljawad@psu.edu.sa}
\and
Raziye Mert$^{b}$\\
\texttt{rmert@thk.edu.tr}
\and
Delfim F. M. Torres$^{c}$\\
\texttt{delfim@ua.pt}}

\date{$^{a}$Department of Mathematics and General Sciences,\\
Prince Sultan University, P. O. Box 66833, 11586 Riyadh, Saudi Arabia\\[0.3cm]
$^{b}$Mechatronic Engineering Department,\\ 
University of Turkish Aeronautical Association, 06790 Ankara, Turkey\\[0.3cm]
$^{c}$Center for Research and Development in Mathematics and Applications (CIDMA),\\ 
Department of Mathematics, University of Aveiro, 3810-193 Aveiro, Portugal}

\maketitle


\begin{abstract}
We introduce new fractional operators of variable order on isolated time scales
with Mittag--Leffler kernels. This allows a general formulation 
of a class of fractional variational problems involving 
variable-order difference operators. Main results give 
fractional integration by parts formulas and necessary 
optimality conditions of Euler--Lagrange type.

\bigskip

\noindent {\bf Keywords:} 
fractional calculus on isolated time scales;
variable order operators with Mittag--Leffler kernels;
fractional sums and differences of variable order;
summation by parts;
variational principles on isolated time scales.

\medskip

\noindent {\bf MSC 2010:} 26A33; 26E70; 49K05.
\end{abstract}


\section{Introduction}
\label{s:1}

Fractional calculus is a generalization of ordinary differentiation 
and integration to an arbitrary non-integer order. It has been used 
effectively in the modeling of many problems in various fields of science and engineering,
reflecting successfully the description of non-local properties of complex systems 
\cite{Dumitru,Samko}. For the sake of finding more fractional operators with different kernels, 
recently several authors have introduced and studied new non-local derivatives with non-singular 
kernels and have applied them successfully to some real world problems 
\cite{Thabet1,Thabet2,Abdon,CF,CF1,LN}. What makes those fractional derivatives 
with Mittag--Leffler kernels more interesting is that their corresponding fractional 
integrals contain Riemann--Liouville fractional integrals as part of their structure. 
Moreover, such operators enable numerical analysts to develop more efficient algorithms 
in solving fractional dynamical systems by concentrating only on the coefficients 
of the differential equations rather than worrying about the singularity of the kernels, 
as in the case of classical fractional operators \cite{MR3726912}.

In 1993, Samko and Ross investigated integrals and derivatives not of 
a constant but of variable order \cite{RS,S,SR}. 
Afterwards, several pure mathematical and applicational papers contributed 
to the theory of variable order fractional calculus  
\cite{AP,C,DC,LH,RC1,RC2}. Here we continue this line of research.

The article is organized as follows. In Section~\ref{sec:02}, 
we introduce new definitions of two different types of left and right 
nabla fractional sums of variable order, two different types 
of discrete versions of the left and right generalized fractional 
integral operators, together with two different types of 
fractional sums and differences of variable order
in the sense of Atangana--Baleanu. 
Afterwards, in Section~\ref{sec:03}, we prove integration by parts formulas
for Atangana--Baleanu fractional sums and differences with variable order.
We end with Section~\ref{sec:03}, applying our results to the calculus of variations.


\section{Fractional sums and differences of variable order}
\label{sec:02}

The study of fractional calculus on time scales 
was initiated with the papers \cite{MR2728463,MyID:179,MR2800417}
and is now under strong development: see, e.g., 
\cite{MR3571716,MR3453710,MyID:328,MR3729560,MR3799774}.
Here, inspired by the results of \cite{Thabet2,Abdon},
we introduce new nabla fractional operators of variable order
on isolated time scales. The reader interested
on the motivation and importance to consider variable order operators
is referred to \cite{MR3434464,MR3606691,MR3714436} and references therein.

Let $a,b\in\mathbb{R}$ with $b-a$ a positive integer.
The sets $\mathbb{N}_a$, $_{b}\mathbb{N}$, 
and $\mathbb{N}_{a,b}$ are defined by
$$
\mathbb{N}_a=\{a,a+1,a+2,...\},\quad _{b}\mathbb{N}=\{...,b-2,b-1,b\},
\quad \mathbb{N}_{a,b}=\{a,a+1,a+2,...,b\},
$$
respectively. Our operators use the concepts 
of \emph{rising function} and 
\emph{discrete Mittag--Leffler function}.

\begin{definition}[Rising function \cite{CP}] 
\label{rising}
\noindent (i) For a natural number $m$ and $t\in\mathbb{R}$, 
the $m$ rising (ascending) factorial of $t$ is defined by
\begin{equation*}
t^{\overline{m}}= \prod_{k=0}^{m-1}(t+k),~~~t^{\overline{0}}=1.
\end{equation*}
\noindent (ii) For any real number $\alpha$, 
the (generalized) rising function is defined by
\begin{equation*}\label{alpharising}
t^{\overline{\alpha}}=\frac{\Gamma(t+\alpha)}{\Gamma(t)},
\quad t \in \mathbb{R}\setminus \{\ldots,-2,-1,0\},
\quad 0^{\overline{\alpha}}=0.
\end{equation*}
\end{definition}

\begin{definition}[Nabla discrete Mittag--Leffler function \cite{dualCaputo,Thsemi}]
\label{nDML} 
For $\lambda \in \mathbb{R}$, $|\lambda|<1$ and $\alpha, \beta, z \in \mathbb{C}$ 
with $Re(\alpha)>0$, the nabla discrete Mittag--Leffler function is defined by
\begin{equation*}
E_{\overline{\alpha, \beta}}(\lambda,z)
= \sum_{k=0}^\infty \lambda^k
\frac{z^{\overline{k\alpha+\beta-1}}} {\Gamma(\alpha k+\beta)}.
\end{equation*}
For $\beta=1$, we write
\begin{equation*} 
E_{\overline{\alpha}}(\lambda, z)
\triangleq E_{\overline{\alpha, 1}}(\lambda, z)
= \sum_{k=0}^\infty \lambda^k
\frac{z^{\overline{k\alpha}}} {\Gamma(\alpha k+1)}.
\end{equation*}
\end{definition}

To start, we define two different types 
of nabla fractional sums of variable order.

\begin{definition}[Left nabla fractional sums of order $\alpha(t)$ --- types $I$ and $II$] 
\label{left fractional sums}
Let $0<\alpha(t)\leq 1$ for all $t\in\mathbb{N}_a$. For a function 
$f:\mathbb{N}_a\rightarrow \mathbb{R}$,
\begin{enumerate}
\item the type $I$ left nabla fractional sum of order $\alpha(t)$ is defined by
\begin{equation*}
_{a}\nabla^{-\alpha(t)}f(t)=\frac{1}{\Gamma(\alpha(t))}
\sum_{s=a+1}^t(t-\rho(s))^{\overline{\alpha(t)-1}}f(s),
\quad t \in \mathbb{N}_{a+1};
\end{equation*}
\item the type $II$ left nabla fractional sum of order $\alpha(t)$ is defined by
\begin{equation*}
_{a}^{*}\nabla^{-\alpha(t)}f(t)
=\sum_{s=a+1}^t(t-\rho(s))^{\overline{\alpha(s)-1}}f(s)
\frac{1}{\Gamma(\alpha(s))},\quad t \in \mathbb{N}_{a+1}.
\end{equation*}
\end{enumerate}
\end{definition}

\begin{definition}[Right nabla fractional sums of order $\alpha(t)$ --- types $I$ and $II$] 
\label{right fractional sums}
Let $0<\alpha(t)\leq 1$ for all $t\in{_{b}\mathbb{N}}$. 
For a function $f:{_{b}\mathbb{N}}\rightarrow \mathbb{R}$,
\begin{enumerate}
\item the type $I$ right nabla fractional sum of order $\alpha(t)$ is defined by
\begin{equation*}
\nabla_{b}^{-\alpha(t)} f(t)
=\frac{1}{\Gamma(\alpha(t))}
\sum_{s=t}^{b-1}(s-\rho(t))^{\overline{\alpha(t)-1}}f(s),\quad t \in {_{b-1}\mathbb{N}};
\end{equation*}
\item the type $II$ right nabla fractional sum of order $\alpha(t)$ is defined by
\begin{equation*}
^{*}\nabla_{b}^{-\alpha(t)} f(t)
= \sum_{s=t}^{b-1}(s-\rho(t))^{\overline{\alpha(s)-1}}f(s)
\frac{1}{\Gamma(\alpha(s))},\quad t \in {_{b-1}\mathbb{N}}.
\end{equation*}
\end{enumerate}
\end{definition}

Following \cite{Thabet2}, we now define two different discrete versions 
of the left and right generalized fractional integral operators.

\begin{definition}[Discrete left generalized fractional integral operators --- types $I$ and $II$]
\label{GIODEF1}
Let  $0<\alpha(t)<1/2$ for all $t\in\mathbb{N}_a$. For a function 
$\varphi:\mathbb{N}_a\rightarrow \mathbb{R}$,
\begin{enumerate}
\item the type $I$ discrete left generalized fractional integral operator is defined by
\begin{equation}
\label{GIO}
\textbf{E}_{\overline{\alpha(t),1},\frac{-\alpha(t)}{1-\alpha(t)},a^+}\varphi(t)
=\frac{B(\alpha(t))}{1-\alpha(t)}\sum_{s=a+1}^t E_{\overline{\alpha(t)}} 
\left[\frac{-\alpha(t)}{1-\alpha(t)}, t-\rho(s)\right]\varphi(s),
\quad t\in \mathbb{N}_{a+1};
\end{equation}
\item the type $II$ discrete left generalized fractional integral operator is defined by
\begin{equation}
\label{starGIO}
\mathcal{E}_{\overline{\alpha(t),1},\frac{-\alpha(t)}{1-\alpha(t)},a^+}\varphi(t)
=\sum_{s=a+1}^t \frac{B(\alpha(s))}{1-\alpha(s)} E_{\overline{\alpha(s)}} 
\left[\frac{-\alpha(s)}{1-\alpha(s)},t-\rho(s)\right]\varphi(s),\quad t\in \mathbb{N}_{a+1}.
\end{equation}
\end{enumerate}
\end{definition}

\begin{definition}[Discrete right generalized fractional integral operators --- types $I$ and $II$]
\label{GIODEF2}
Let  $0<\alpha(t)<1/2$ for all $t\in{_{b}\mathbb{N}}$. 
For a function $\varphi:{_{b}\mathbb{N}}\rightarrow \mathbb{R}$,
\begin{enumerate}
\item the type $I$ discrete right generalized fractional integral operator is defined by
\begin{equation}
\label{GIOr}
\textbf{ E}_{\overline{\alpha(t),1},\frac{-\alpha(t)}{1-\alpha(t)} ,b^-}\varphi(t)
=\frac{B(\alpha(t))}{1-\alpha(t)}\sum_{s=t}^{b-1}  E_{\overline{\alpha(t)}}\left[
\frac{-\alpha(t)}{1-\alpha(t)}, s-\rho(t)\right]\varphi(s),
\quad t\in {_{b-1}\mathbb{N}};
\end{equation}
\item the type $II$ discrete  right generalized fractional integral operator is defined by
\begin{equation}
\label{starGIOr}
\mathcal{E}_{\overline{\alpha(t),1}, \frac{-\alpha(t)}{1-\alpha(t)},b^-}\varphi(t)
= \sum_{s=t}^{b-1} \frac{B(\alpha(s))}{1-\alpha(s)} E_{\overline{\alpha(s)}} 
\left[\frac{-\alpha(s)}{1-\alpha(s)},s-\rho(t)\right]\varphi(s),
\quad t\in {_{b-1}\mathbb{N}}.
\end{equation}
\end{enumerate}
\end{definition}

We now define two different types of fractional sums and differences 
of variable order in the sense of Atangana--Baleanu \cite{Abdon}
(the so-called $AB$ operators).

\begin{definition}[Left $AB$ nabla fractional sums of order $\alpha(t)$ --- types $I$ and $II$]
\label{left AB fractional sums}
Let $0<\alpha(t)\leq 1$ for all $t\in\mathbb{N}_a$. 
For a function $f:\mathbb{N}_a\rightarrow \mathbb{R}$,
\begin{enumerate}
\item the type $I$ left $AB$ nabla fractional sum of order $\alpha(t)$ is defined by
\begin{equation}
\label{EQ5}
\begin{split}
^{AB}_{a}\nabla^{-\alpha(t)}f(t)
&=\frac{1-\alpha(t)}{B(\alpha(t))}f(t)
+\frac{\alpha(t)}{B(\alpha(t))\Gamma(\alpha(t))}
\sum_{s=a+1}^t(t-\rho(s))^{\overline{\alpha(t)-1}}f(s)\\
&=\frac{1-\alpha(t)}{B(\alpha(t))}f(t)+\frac{\alpha(t)}{B(\alpha(t))}{_{a}}
\nabla^{-\alpha(t)}f(t),\quad t \in \mathbb{N}_{a+1};
\end{split}
\end{equation}
\item the type $II$ left $AB$ nabla fractional sum of order $\alpha(t)$ is defined by
\begin{equation}
\label{EQ7}
\begin{split}
^{*AB}_{a}\nabla^{-\alpha(t)}f(t)
&=\frac{1-\alpha(t)}{B(\alpha(t))}f(t)+\sum_{s=a+1}^t\frac{\alpha(s)}{B(\alpha(s))
\Gamma(\alpha(s))}(t-\rho(s))^{\overline{\alpha(s)-1}}f(s)\\
&=\frac{1-\alpha(t)}{B(\alpha(t))}f(t)+_{a}^{*}\nabla^{-\alpha(t)}
\frac{\alpha f}{B\circ\alpha}(t),\quad t \in \mathbb{N}_{a+1}.
\end{split}
\end{equation}
\end{enumerate}
\end{definition}

\begin{definition}[Right $AB$ nabla fractional sums of order $\alpha(t)$ --- types $I$ and $II$]
\label{right AB fractional sums}
Let $0<\alpha(t)\leq 1$ for all $t\in{_{b}\mathbb{N}}$. 
For a function $f:{_{b}\mathbb{N}}\rightarrow \mathbb{R}$,
\begin{enumerate}
\item the type $I$ right $AB$ nabla fractional 
sum of order $\alpha(t)$ is defined by
\begin{equation}
\label{EQ6}
\begin{split}
^{AB}\nabla_{b}^{-\alpha(t)}f(t)
&=\frac{1-\alpha(t)}{B(\alpha(t))}f(t)+\frac{\alpha(t)}{B(\alpha(t))
\Gamma(\alpha(t))}\sum_{s=t}^{b-1}(s-\rho(t))^{\overline{\alpha(t)-1}}f(s)\\
&=\frac{1-\alpha(t)}{B(\alpha(t))}f(t)+\frac{\alpha(t)}{B(\alpha(t))}
\nabla_{b}^{-\alpha(t)}f(t),\quad t \in {_{b-1}\mathbb{N}};
\end{split}
\end{equation}
\item the type $II$ right $AB$ nabla fractional sum of order $\alpha(t)$ is defined by
\begin{equation}
\label{EQ8}
\begin{split}
^{*AB}\nabla_{b}^{-\alpha(t)}f(t)
&=\frac{1-\alpha(t)}{B(\alpha(t))}f(t)+\sum_{s=t}^{b-1}\frac{\alpha(s)}{B(\alpha(s))
\Gamma(\alpha(s))}(s-\rho(t))^{\overline{\alpha(s)-1}}f(s)\\
&=\frac{1-\alpha(t)}{B(\alpha(t))}f(t)+^{*}\nabla_{b}^{-\alpha(t)}
\frac{\alpha f}{B\circ\alpha}(t),\quad t \in {_{b-1}\mathbb{N}}.
\end{split}
\end{equation}
\end{enumerate}
\end{definition}

Note that in Definitions~\ref{left AB fractional sums} 
and \ref{right AB fractional sums}, if $\alpha(t) \equiv 0$, then we recover
the initial function; if $\alpha(t) \equiv 1 $, then we recover the ordinary sum.

\begin{definition}[Left Riemann--Liouville $AB$ nabla fractional differences 
of order $\alpha(t)$ --- types I and II]
\label{RLVO}
Let $0<\alpha(t)<1/2$ for all $t\in\mathbb{N}_a$. 
For a function $f:\mathbb{N}_a\rightarrow \mathbb{R}$,
\begin{enumerate}
\item the type $I$ left Riemann--Liouville $AB$ nabla fractional 
difference of order  $\alpha(t)$ is defined by
\begin{equation}
\label{dq}
^{ABR}_{a}\nabla^{\alpha(t)}f(t)=\nabla\textbf{E}_{\overline{\alpha(t),1}, 
\frac{-\alpha(t)}{1-\alpha(t)},a^+}f(t),\quad t\in\mathbb{N}_{a+1};
\end{equation}
\item the type $II$ left Riemann--Liouville $AB$ nabla fractional 
difference of order  $\alpha(t)$ is defined by
\begin{equation}
\label{dqq}
^{ABR}_{a}\widehat{\nabla}^{\alpha(t)}f(t)=\nabla\mathcal{E}_{\overline{\alpha(t),1}, 
\frac{-\alpha(t)}{1-\alpha(t)},a^+}f(t),\quad t\in\mathbb{N}_{a+1}.
\end{equation}
\end{enumerate}
\end{definition}

\begin{definition}[Right Riemann--Liouville $AB$ nabla fractional differences 
of order $\alpha(t)$ --- types I and II]
\label{RLVOR}
Let $0<\alpha(t)<1/2$ for all $t\in{_{b}\mathbb{N}}$. 
For a function $f:{_{b}\mathbb{N}}\rightarrow \mathbb{R}$,
\begin{enumerate}
\item the type $I$ right Riemann--Liouville $AB$ nabla 
fractional difference of order  $\alpha(t)$ is defined by
\begin{equation*}
^{ABR}\nabla_b^{\alpha(t)}f(t)=-\Delta\textbf{E}_{\overline{\alpha(t),1}, 
\frac{-\alpha(t)}{1-\alpha(t)},b^-}f(t),\quad t\in {_{b-1}\mathbb{N}};
\end{equation*}
\item the type $II$ right Riemann--Liouville $AB$ 
nabla fractional difference of order  $\alpha(t)$ is defined by
\begin{equation*}
^{ABR}\widehat{\nabla}_b^{\alpha(t)}f(t)=-\Delta\mathcal{E}_{\overline{\alpha(t),1}, 
\frac{-\alpha(t)}{1-\alpha(t)},b^-}f(t),\quad t\in {_{b-1}\mathbb{N}}.
\end{equation*}
\end{enumerate}
\end{definition}

\begin{definition}[Left Caputo $AB$ nabla fractional differences of order $\alpha(t)$ --- types I and II]
\label{LeftC}
Let $0<\alpha(t)<1/2$ for all $t\in\mathbb{N}_a$. For a function $f:\mathbb{N}_a\rightarrow \mathbb{R},$
\begin{enumerate}
\item the type $I$ left Caputo $AB$ nabla fractional difference of order $\alpha(t)$ is defined by
\begin{equation*}
^{ABC}_{a}\nabla^{\alpha(t)}f(t)=\textbf{ E}_{\overline{\alpha(t),1},
\frac{-\alpha(t)}{1-\alpha(t)},a^+}\nabla f(t),\quad t\in\mathbb{N}_{a+1};
\end{equation*}
\item the type $II$ left Caputo $AB$ nabla fractional 
difference of order  $\alpha(t)$ is defined by
\begin{equation*}
^{ABC}_{a}\widehat{\nabla}^{\alpha(t)}f(t)=\mathcal{ E}_{
\overline{\alpha(t),1},\frac{-\alpha(t)}{1-\alpha(t)},a^+}
\nabla f(t),\quad t\in\mathbb{N}_{a+1}.
\end{equation*}
\end{enumerate}
\end{definition}

\begin{definition}[Right Caputo $AB$ nabla fractional 
differences of order $\alpha(t)$ --- types I and II]
\label{RightC}
Let $0<\alpha(t)<1/2$ for all $t\in{_{b}\mathbb{N}}$. 
For a function $f:{_{b}\mathbb{N}}\rightarrow \mathbb{R}$,
\begin{enumerate}
\item the type $I$ right Caputo $AB$ nabla fractional 
difference of order  $\alpha(t)$ is defined by
\begin{equation*}
^{ABC}\nabla^{\alpha(t)}_{b}f(t)=-\textbf{E}_{\overline{\alpha(t),1},
\frac{-\alpha(t)}{1-\alpha(t)},b^-}\Delta f (t),\quad t\in {_{b-1}\mathbb{N}};
\end{equation*}
\item the type $II$ right Caputo $AB$ nabla fractional 
difference of order  $\alpha(t)$ is defined by
\begin{equation*}
^{ABC}\widehat{\nabla}^{\alpha(t)}_{b}f(t)=-\mathcal{E}_{\overline{\alpha(t),1},
\frac{-\alpha(t)}{1-\alpha(t)},b^-}\Delta f(t),\quad t\in {_{b-1}\mathbb{N}}.
\end{equation*}
\end{enumerate}
\end{definition}

\begin{remark}
If we replace $\alpha(t)$ in (\ref{GIO}) and (\ref{GIOr}) by $\alpha(t-s)$ 
and replace each $\alpha(s)$ in (\ref{starGIO}) and (\ref{starGIOr}) by  
$\alpha(t-s)$, then the $ABR$ and $ABC$ fractional differences with variable 
order can be expressed in convolution  form. Similarly, if we replace $\alpha(t)$ 
in (\ref{EQ5}) and (\ref{EQ6}) by $\alpha(t-s)$ and replace each $\alpha(s)$ 
in (\ref{EQ7}) and (\ref{EQ8}) by  $\alpha(t-s)$, then the second part of the 
$AB$ fractional integrals with variable order can be expressed in convolution form.
\end{remark}


\section{Summation by parts for variable order fractional operators}
\label{sec:03}

Summation/integration by parts has a very important role in mathematics:
see, e.g., \cite{MR3749950,MR1572851,MR3762312}. This is particularly
true in the calculus of variations and optimal control, 
to prove necessary optimality conditions of Euler--Lagrange type 
(cf. proof of Theorem~\ref{thm:EL:eq}).

\begin{lemma}[Integration by parts formula for nabla fractional sums of order $\alpha(t)$]	
\label{IBPFD} Let $0<\alpha(t)\leq 1$ for all $t\in\mathbb{N}_{a,b}$.  
For functions $f,g:\mathbb{N}_{a,b}\rightarrow \mathbb{R}$, we have
$$
\sum_{t=a+1}^{b-1}f(t)~_{a}\nabla^{-\alpha(t)}g(t)
=\sum_{t=a+1}^{b-1}g(t)~^{*}\nabla_{b}^{-\alpha(t)}f(t);
$$
$$
\sum_{t=a+1}^{b-1}f(t)~\nabla_{b}^{-\alpha(t)}g(t)
=\sum_{t=a+1}^{b-1}g(t)~_{a}^{*}\nabla^{-\alpha(t)}f(t).
$$
\end{lemma}

\begin{proof}
From Definition~\ref{left fractional sums}, and by changing 
the order of summation, we get
\begin{eqnarray*}
\sum_{t=a+1}^{b-1}f(t)~_{a}\nabla^{-\alpha(t)}g(t)
&=&\sum_{t=a+1}^{b-1}f(t)\,\frac{1}{\Gamma(\alpha(t))}
\sum_{s=a+1}^t(t-\rho(s))^{\overline{\alpha(t)-1}}g(s)\\
&=&\sum_{s=a+1}^{b-1}g(s)\left(\sum_{t=s}^{b-1}(t
-\rho(s))^{\overline{\alpha(t)-1}}f(t)\frac{1}{\Gamma(\alpha(t))}\right)\\
&=& \sum_{s=a+1}^{b-1}g(s)~^{*}\nabla_{b}^{-\alpha(t)}f(s).
\end{eqnarray*}
The proof of the second assertion follows similarly.
\end{proof}

Now, with the help of Lemma~\ref{IBPFD}, we can prove the following 
integration by parts formula for $AB$ fractional sums of variable order.

\begin{theorem}[Integration by parts formula for $AB$ nabla fractional sums of order $\alpha(t)$] 
\label{T2}
Let $0<\alpha(t)\leq 1$ for all $t\in\mathbb{N}_{a,b}$.  
For functions $f,g:\mathbb{N}_{a,b}\rightarrow \mathbb{R}$, we have
\begin{equation*}
\sum_{t=a+1}^{b-1}f(t)~_{a}^{AB}\nabla^{-\alpha(t)}g(t)
=\sum_{t=a+1}^{b-1}g(t)~^{*AB}\nabla_{b}^{-\alpha(t)}f(t);
\end{equation*}
\begin{equation*}
\sum_{t=a+1}^{b-1}f(t)~_{a}^{*AB}\nabla^{-\alpha(t)}g(t)
=\sum_{t=a+1}^{b-1}g(t)~^{AB}\nabla_{b}^{-\alpha(t)}f(t).
\end{equation*}
\end{theorem}

\begin{proof}
From  Definition~\ref{left AB fractional sums} 
and the first part of Lemma~\ref{IBPFD}, we get
\begin{eqnarray*}
\sum_{t=a+1}^{b-1}f(t)~_{a}^{AB}\nabla^{-\alpha(t)}g(t)
&=&\sum_{t=a+1}^{b-1}f(t)\frac{1-\alpha(t)}{B(\alpha(t))}g(t)
+\sum_{t=a+1}^{b-1}f(t)\frac{\alpha(t)}{B(\alpha(t))}{_{a}}\nabla^{-\alpha(t)}g(t)\\
&=&\sum_{t=a+1}^{b-1}f(t)\frac{1-\alpha(t)}{B(\alpha(t))}g(t)
+\sum_{t=a+1}^{b-1}g(t)~^{*}\nabla_{b}^{-\alpha(t)}\frac{\alpha f}{B\circ\alpha}(t)\\
&=&\sum_{t=a+1}^{b-1}g(t)\left(\frac{1-\alpha(t)}{B(\alpha(t))}f(t)
+^{*}\nabla_{b}^{-\alpha(t)}\frac{\alpha f}{B\circ\alpha}(t)\right)\\
&=&\sum_{t=a+1}^{b-1}g(t)~^{*AB}\nabla_{b}^{-\alpha(t)}f(t).
\end{eqnarray*}
The proof of the second assertion is similar to the first one. It follows from 
Definition~\ref{left AB fractional sums} and the second part of Lemma~\ref{IBPFD}.
\end{proof}

\begin{lemma}
\label{GIOIBPF}
Let $0<\alpha(t)<1/2$ for all $t\in\mathbb{N}_{a,b}$.  
For functions $f,g:\mathbb{N}_{a,b}\rightarrow \mathbb{R}$, we have
\begin{equation*}
\sum_{t=a+1}^{b-1}f(t)~\textbf{E}_{\overline{\alpha(t),1},
\frac{-\alpha(t)}{1-\alpha(t)},a^+}g(t)=\sum_{t=a+1}^{b-1} 
g(t)~\mathcal{E}_{\overline{\alpha(t),1},\frac{-\alpha(t)}{1-\alpha(t)},b^-}f(t);
\end{equation*}
\begin{equation*}
\sum_{t=a+1}^{b-1}f(t)~\mathcal{E}_{\overline{\alpha(t),1},
\frac{-\alpha(t)}{1-\alpha(t)},a^+}g(t)=\sum_{t=a+1}^{b-1} g(t)
~\textbf{E}_{\overline{\alpha(t),1},\frac{-\alpha(t)}{1-\alpha(t)},b^-}f(t).
\end{equation*}
\end{lemma}

\begin{proof}
From Definitions~\ref{GIODEF1} and \ref{GIODEF2}, 
and by changing the order of summation, we have
\begin{eqnarray*}
\sum_{t=a+1}^{b-1}f(t)~\textbf{E}_{\overline{\alpha(t),1},
\frac{-\alpha(t)}{1-\alpha(t)},a^+}g(t)
&=&\sum_{t=a+1}^{b-1} f(t)\frac{B(\alpha(t))}{1-\alpha(t)}
\sum_{s=a+1}^tE_{\overline{\alpha(t)}} 
\left[\frac{-\alpha(t)}{1-\alpha(t)},t-\rho(s)\right]g(s)\\
&=&\sum_{s=a+1}^{b-1}g(s)\sum_{t=s}^{b-1}
\frac{B(\alpha(t))}{1-\alpha(t)}~E_{\overline{\alpha(t)}} 
\left[\frac{-\alpha(t)}{1-\alpha(t)}, t-\rho(s)\right]f(t)\\
&=&\sum_{s=a+1}^{b-1}g(s)~\mathcal{E}_{\overline{\alpha(s), 1},
\frac{-\alpha(s)}{1-\alpha(s)},b^-}f(s).
\end{eqnarray*}
The proof of the second assertion follows similarly.
\end{proof}

\begin{theorem}
\label{main}
Let $0<\alpha(t)<1/2$ for all $t\in\mathbb{N}_{a,b}$.  
For functions $f,g:\mathbb{N}_{a,b}\rightarrow \mathbb{R}$, we have
\begin{equation*}
\sum_{t=a+1}^{b-1}f(t)~^{ABC}_{a}\nabla^{\alpha(t)}g(t)
=g(t)~\mathcal{E}_{\overline{\alpha(t), 1},\frac{-\alpha(t)}{1-\alpha(t)},b^-}
f(t)\Big|_{a}^{b-1}+\sum_{t=a+1}^{b-1}g(t-1)~
^{ABR}\widehat{\nabla}_b^{\alpha(t)}f(t-1);
\end{equation*}
\begin{equation*}
\sum_{t=a+1}^{b-1}f(t)~^{ABC}_{a}\widehat{\nabla}^{\alpha(t)}g(t)
=g(t)~\textbf{E}_{\overline{\alpha(t), 1},\frac{-\alpha(t)}{1-\alpha(t)},b^-}
f(t)\Big|_{a}^{b-1}+\sum_{t=a+1}^{b-1}g(t-1)~
^{ABR}\nabla_b^{\alpha(t)}f(t-1);
\end{equation*}
\begin{equation*}
\sum_{t=a+1}^{b-1}f(t)~^{ABC}\nabla^{\alpha(t)}_{b}g(t)
=-g(t)~\mathcal{E}_{\overline{\alpha(t), 1},
\frac{-\alpha(t)}{1-\alpha(t)},a^+}f(t)\Big|_{a+1}^{b}
+\sum_{t=a+1}^{b-1}g(t+1)~^{ABR}_{a}\widehat{\nabla}^{\alpha(t)}f(t+1);
\end{equation*}
\begin{equation*}
\sum_{t=a+1}^{b-1}f(t)~^{ABC}\widehat{\nabla}^{\alpha(t)}_{b}g(t)
=-g(t)~\textbf{E}_{\overline{\alpha(t), 1},
\frac{-\alpha(t)}{1-\alpha(t)},a^+}f(t)\Big|_{a+1}^{b}
+\sum_{t=a+1}^{b-1}g(t+1)~^{ABR}_{a}\nabla^{\alpha(t)}f(t+1).
\end{equation*}
\end{theorem}

\begin{proof}
We will only prove the first assertion. The proof of the others follow similarly. 
From Definitions~\ref{RLVOR} and \ref{LeftC}, the first part of Lemma~\ref{GIOIBPF} 
and the summation by parts formula from ordinary difference calculus, we get
\begin{equation*}
\begin{split}
\sum_{t=a+1}^{b-1}&f(t)~^{ABC}_{a}\nabla^{\alpha(t)}g(t)\\
&=\sum_{t=a+1}^{b-1}f(t)\textbf{ E}_{\overline{\alpha(t),1},
\frac{-\alpha(t)}{1-\alpha(t)},a^+}\nabla g(t)\\
&=\sum_{t=a+1}^{b-1} \nabla g(t)~\mathcal{E}_{\overline{\alpha(t), 1},
\frac{-\alpha(t)}{1-\alpha(t)},b^-}f(t)\\
&=g(t)~\mathcal{E}_{\overline{\alpha(t), 1},
\frac{-\alpha(t)}{1-\alpha(t)},b^-}f(t)|_{a}^{b-1}
-\sum_{t=a+1}^{b-1}g(t-1)\nabla\mathcal{E}_{\overline{\alpha(t), 1},
\frac{-\alpha(t)}{1-\alpha(t)},b^-}f(t)\\
&=g(t)~\mathcal{E}_{\overline{\alpha(t), 1},\frac{-\alpha(t)}{1-\alpha(t)},b^-}f(t)|_{a}^{b-1}
-\sum_{t=a+1}^{b-1}g(t-1)\Delta\mathcal{E}_{\overline{\alpha(t), 1},
\frac{-\alpha(t)}{1-\alpha(t)},b^-}f(t-1)\\
&=g(t)~\mathcal{E}_{\overline{\alpha(t), 1},
\frac{-\alpha(t)}{1-\alpha(t)},b^-}f(t)|_{a}^{b-1}
+\sum_{t=a+1}^{b-1}g(t-1)~
^{ABR}\widehat{\nabla}_b^{\alpha(t)}f(t-1).
\end{split}
\end{equation*}
The proof is complete.
\end{proof}


\section{Variable order fractional variational principles}
\label{sec:04}

The fractional calculus of variations of variable-order
is a subject under strong current development \cite{MyID:404,MR3787702}.
However, to the best of our knowledge, available results 
are only for the continuous time scale $\mathbb{T} = \mathbb{R}$.
Here we obtain the main result of a variational calculus, that is,
an Euler--Lagrange necessary optimality condition, for the 
isolated time scale $\mathbb{T} = \mathbb{N}_{a+1,b-1}$.

Let $J$ be a functional of the form
\begin{equation*}
J(f)=\sum_{t=a+1}^{b-1}L(t,f^{\rho}(t),~^{ABC}_{a}\nabla^{\alpha(t)}f(t)),
\end{equation*}
where $0<\alpha(t)<1/2$ for all $t\in\mathbb{N}_{a+1,b-1}$, 
$f:\mathbb{N}_{a,b-1}\to\mathbb{R}$ and 
$L:\mathbb{N}_{a+1,b-1}\times\mathbb{R}\times\mathbb{R}\to\mathbb{R}$.

\begin{theorem}
\label{thm:EL:eq}
Let $f$ be a local extremum of $J$ satisfying the boundary conditions
\begin{equation*}
f(a)=A,\quad f(b-1)=B.
\end{equation*}
Then $f$ satisfies the Euler--Lagrange equation
$$
L_1^{\sigma}(t)+~^{ABR}\widehat{\nabla}_b^{\alpha(t)}L_2(t)=0,
\quad t\in\mathbb{N}_{a+1,b-2},
$$
where $L_1=\frac{\partial L}{\partial f^{\rho}}$ and 
$L_2=\frac{\partial L}{\partial~^{ABC}_{a}\nabla^{\alpha(t)}f}$.
\end{theorem}

\begin{proof}
Let $\varepsilon$ be a small real parameter and $\eta:\mathbb{N}_{a,b-1}\to\mathbb{R}$ 
be a function such that $\eta(a)=\eta(b-1)=0$. Consider a variation of $f$, 
say $f+\varepsilon\eta$. Since the Caputo difference operator $
^{ABC}_{a}\nabla^{\alpha(t)}$ is linear, it follows that
\begin{equation*}
J(f+\varepsilon\eta)=\sum_{t=a+1}^{b-1}L(t,f^{\rho}(t)
+\varepsilon\eta^{\rho}(t),~^{ABC}_{a}\nabla^{\alpha(t)}f(t)
+\varepsilon~^{ABC}_{a}\nabla^{\alpha(t)}\eta(t)).
\end{equation*}
Define $\hat{J}(\varepsilon)=J(f+\varepsilon\eta)$. Because $f$ 
is a local extremizer of $J$, $\hat{J}$ attains a local extremum 
at $\varepsilon= 0$. Differentiating $\hat{J}(\varepsilon)$ at zero, we get
\begin{equation*}
\sum_{t=a+1}^{b-1}\eta^{\rho}(t)\frac{\partial L}{\partial
f^\rho}(t,f^{\rho}(t),~^{ABC}_{a}\nabla^{\alpha(t)}f(t))
+~^{ABC}_{a}\nabla^{\alpha(t)}\eta(t)~\frac{\partial L}{\partial~ 
^{ABC}_{a}\nabla^{\alpha(t)}f}(t,f^{\rho}(t),~^{ABC}_{a}\nabla^{\alpha(t)}f(t))=0.
\end{equation*}
Using the first integration by parts formula in Theorem~\ref{main}, we have
\begin{multline*}
\sum_{t=a+1}^{b-1}\eta^{\rho}(t)\left[\frac{\partial L}{\partial
f^\rho}(t,f^{\rho}(t),~^{ABC}_{a}\nabla^{\alpha(t)}f(t))
+\left(^{ABR}\widehat{\nabla}_b^{\alpha(t)}
\frac{\partial L}{\partial~ ^{ABC}_{a}
\nabla^{\alpha(t)}f}(t,f^{\rho}(t),~^{ABC}_{a}
\nabla^{\alpha(t)}f(t))\right)(t-1)\right]\\
+\eta(t)\left(\mathcal{E}_{\overline{\alpha(t), 1},
\frac{-\alpha(t)}{1-\alpha(t)},b^-}\frac{\partial L}{\partial~
^{ABC}_{a}\nabla^{\alpha(t)}f}(t,f^{\rho}(t),~^{ABC}_{a}
\nabla^{\alpha(t)}f(t))\right)(t)\Big|_{a}^{b-1}=0.
\end{multline*}
Since $\eta(a)=\eta(b-1)=0$ and $\eta$ is arbitrary, it follows that
$$
\frac{\partial L}{\partial f^\rho}(t,f^{\rho}(t),~^{ABC}_{a}\nabla^{\alpha(t)}f(t))
+\left(^{ABR}\widehat{\nabla}_b^{\alpha(t)}\frac{\partial L}{\partial~
^{ABC}_{a}\nabla^{\alpha(t)}f}(t,f^{\rho}(t),~^{ABC}_{a}
\nabla^{\alpha(t)}f(t))\right)(t-1)=0
$$
for all $t\in\mathbb{N}_{a+2,b-1}$.
\end{proof}

Although we only consider here a class of 
fractional variable order variational problems (FVOVP),
our Theorem~\ref{thm:EL:eq} can be easily extended to many other related
FVOVPs involving the new variable-order fractional differences
introduced in Section~\ref{sec:02}. We trust that this observation 
will initiate some interest in further future developments.


\section*{Acknowledgements}

Abdeljawad is grateful to Prince Sultan University for funding 
this work through research group \emph{Nonlinear Analysis Methods in Applied Mathematics} 
(NAMAM), number RG-DES-2017-01-17;
Torres to the support of FCT within the R\&D unit CIDMA, UID/MAT/04106/2013.



\end{document}